\documentclass{article}

\usepackage{PRIMEarxiv}
\usepackage{amsmath}
\usepackage{amssymb}
\usepackage{amsthm}
\usepackage[utf8]{inputenc} 
\usepackage[T1]{fontenc}    
\usepackage{hyperref}       
\usepackage{url}            
\usepackage{booktabs}       
\usepackage{amsfonts}       
\usepackage{nicefrac}       
\usepackage{microtype}      
\usepackage{lipsum}
\usepackage{algorithm,algpseudocode}
\usepackage{fancyhdr}       
\usepackage{graphicx}       
\graphicspath{{media/}}     
\usepackage{cleveref}
\usepackage{color}
\usepackage{algorithmicx}

\newtheorem{definition}{Definition}

\newtheorem{theorem}{Theorem}

\newcommand{\blockcomment}[1]{}
\title{Bravais Lattices for Euclidean Degree Efficient Polynomial Interpolation
}

\author{
  R. Connor Greene \\
  NJIT \\
  Newark, NJ\\
  \texttt{rcg4@njit.edu} \\
}

\begin{document}
\maketitle

\begin{abstract}
A method is presented for forming polynomial interpolants on squares and cubes, which are more efficient in the so-called Euclidean degree than other commonly used methods with the same number of collocation points. These methods have several additional desirable properties. The interpolants can be formed and evaluated via the FFT and have a minimally growing Lebesgue constant. The associated points achieve Gauss-Lobatto order accuracy in integration, out-performing tensor product Gauss-Legendre integration for many $C_\infty$ functions.

This method is related to prior work on total degree efficient collocation points by Yuan Xu et al.\cite{pad_interp_xu_de_marchi, fcc_integrate_xu}
\end{abstract}

\section{Introduction}
Polynomial interpolation is a ubiquitous tool in numerical analysis. This paper will be concerned with Chebyshev polynomial interpolants on domains of the form $[-1, 1]^d$, and in particular the unit square and cube. The application of Chebyshev interpolation in one dimension is well-understood and is a popular basis for spectral methods in the numerical solution of ordinary differential equations ~\cite{ORSZAG198070, Driscoll2014}. In higher dimensions, it is common to use a tensor-product grid of one-dimensional Chebyshev interpolation points. But this practice is not necessarily the most efficient for interpolation and integration; the central idea of the present paper is a re-examination of Chebyshev interpolation methods in light of observations by L. N. Trefethen in the 2017 papers \cite{tref_isotropy} and \cite{trefEucDegree}.

These papers concern the notion of isotropy in collocation points. Often in numerical contexts, the domain of interest represents a subset of some larger space. In this case, there is little reason to expect a function to have a special orientation relative to the domain. Anisotropic collocation points produce orientation dependent truncation errors, which worsens the performance on average. Trefethen goes on to show that for multivariate polynomial interpolation in a rotationally-invariant class of analytic functions, the worst-case truncation error is based on the lowest {\em Euclidean degree} polynomial excluded from the basis. A monomial of the form $\prod_j x_j^{n_j}$ has Euclidean degree $\|\mathbf{n}\|_2$. This stands in contrast to the more commonly used maximal degree and total degree defined, respectively, as $\|\mathbf{n}\|_\infty$ and $\|\mathbf{n}\|_1$. 

The maximal degree basis arises naturally when using tensor product grids for interpolation. Historically, the total degree basis has been studied as an efficient alternative to the maximal degree basis. In two dimensions, there are examples\cite{morrow-patterson-points} of collocation points for integrating total degree polynomials dating back to the Morrow-Patterson points. More recently, the Padua points~\cite{PadOriginal} were the first known example of unisolvent points for interpolating total degree polynomials in 2 dimensions. These points were then explained using ergodic theory as an arc-length uniform sampling along a generating curve.~\cite{pad_interp_generate} This approach has been applied in other contexts as well, but in this paper we will consider an alternative formulation.

The Padua points can be constructed as a mapping of a toroidal lattice to the unit square. This formulation of the Padua points was discovered shortly after the original paper~\cite{pad_interp_xu_de_marchi} and was used to generalize the total degree collocation to higher dimensions~\cite{fcc_integrate_xu,total_degree_cube_demarchi}. It has since been applied in other settings~\cite{Cools2011}. We will apply this framework for finding optimal Euclidean degree collocation points.

Additionally, previous work on this topic focuses on optimizing point count to the exclusion of other considerations. This includes both previous work concerning total degree bases, and more recent work\cite{hecht2020multivariate} in the Euclidean degree. The results of this work will include algorithms for computing interpolant coefficients, derivatives, and integrals with proportionally fewer flops relative to tensor product Chebyshev methods. This means that when using the provided algorithms, one can expect these more efficient collocation points to translate to reductions in computation time.

The present work is divided into four broad sections. The preliminaries will lay out the theoretical concepts underpinning the methods. The Optimal Lattices section will describe the specific lattices considered. The algorithms section will describe how to implement these lattices efficiently. Finally, numerical results are included to show the efficacy of these methods.

In addition to this paper, there is code which implements the algorithms described here along with demonstration code used to produce some of the figures. The full python codebase can be found at \url{https://github.com/rcgreene/ilfft_interp}.

\section{Preliminaries}
This section is mostly review of necessary information for the methods in the later sections. Readers familiar with this material will find new results in \cref{sec:grid-proof} and \cref{sec:comp_lats}. Similar framework for lattice derived collocation points can be found in the work of Yuan Xu.\cite{fcc_integrate_xu}
\subsection{Map from the Circle to the Interval} \label{subsec-map}
Let $\mathbb{S}_1$ denote the circle, defined as a periodic interval $[-\pi, \pi]$.  We will frequently make use of the map $\gamma: \mathbb{S}_1 \rightarrow [-1, 1]: \gamma(\theta) = \cos(\theta)$. This map is surjective to the interval, but is not one-to-one. We will need to make use of a well-defined inverse map $\gamma^{-1}$, for which we define a new domain.

\begin{definition}{Chebyshev Torus}

Let $\mathbb{HT}$ denote the quotient space of $\mathbb{S}_1$ with the additional equivalence $\theta = -\theta$. This modifies the circle by adding symmetry conditions at the points $\theta = 0, \pm n \pi$. We will refer to this domain as the Chebyshev Torus, due to its relationship to Chebyshev polynomials.
\end{definition}

Analytic functions on $\mathbb{S}_1$ can be defined as sums $f_S(\theta) = c_0 + \sum_{k=1}^\infty c_{k} e^{ik\theta} + c_{-k} e^{-ik \theta}$, with $c_{\pm k}$ decaying rapidly with $k$. Analytic functions on $\mathbb{HT}$ can be defined similarly, but with the restriction that $f(-\theta) = f(\theta)$. This is satisfied when $c_{k} = c_{-k}$ for all $k$, which is equivalent to $f_{HT}(\theta) = c_0 + \sum_{k = 1}^{\infty} c_k \cos(k \theta)$. For convenience, we will often consider the $\mathbb{HT}$ basis in terms of the $\mathbb{S}_1$ basis. In such cases, the cardinality of the $\mathbb{HT}$ basis is that of the non-negative indexed components of the associated $\mathbb{S}_1$ basis.
\begin{figure}[H]
\centering
\includegraphics[width=.5\textwidth]{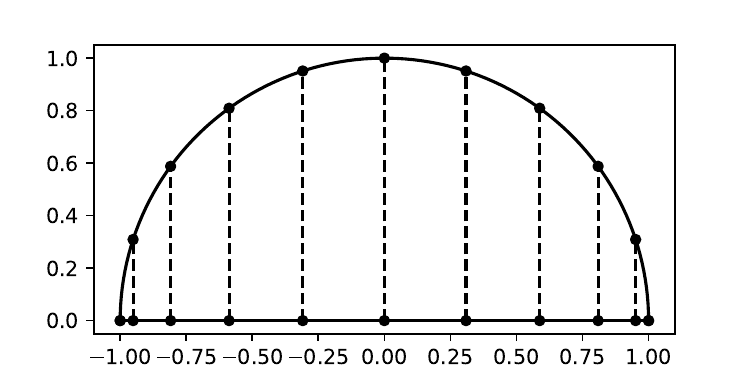}
\caption{11th order Chebyshev points of the Second kind mapped on the interval, and a semi-circular path. Note the uniform distribution of the nodes on the semi-circle.}
\end{figure}
\subsection{Chebyshev Polynomials}

The Chebyshev polynomials of the first kind are the set of orthogonal polynomials on $[-1, 1]$ with respect to the weight $\frac{1}{\sqrt{1 - x^2}}$. They are most easily represented as:
\begin{equation}
T_n(x) = \cos(n \cos^{-1}(x))
\end{equation}
We can see that these functions are polynomials through induction:
\begin{align}
T_0(x) &= \cos(0) = 1 \notag \\
T_1(x) &= \cos(\cos^{-1}(x)) = x \notag \\
\end{align}
And through the application of trigonometric identities, we can derive the recurrence relation: 
\begin{equation}
T_{n+1}(x) = 2xT_{n}(x) - T_{n-1}(x) \notag 
\end{equation}

Further properties of the Chebyshev polynomials can be derived through consideration of the mapping $\gamma$ defined in \cref{subsec-map}. If we take $\theta = \gamma^{-1}(x)$ on $\mathbb{HT}$, then we induce a mapping from the set of Chebyshev polynomials $T_n(x)$ to the cosine basis $\cos(n \theta)$. The orthogonality properties of the Chebyshev polynomials follow easily from considering the $L_2$ inner product of cosines on $\mathbb{HT}$, with $\mu = \frac{1}{\sqrt{1 - x^2}}$ arising from the derivative of $\gamma^{-1}$. 

The most important consequence of this analogy for our purposes comes from considering trapezoidal integration on $\mathbb{S}_1$. It is a well-known fact of approximation theory that the $n$-point trapezoidal rule on a periodic domain is accurate for $e^{ikx}$ with $|k| < n$. This fact is important for interpolation, because it establishes a parity between the integration rule and interpolation at the point set. If $f$ is a function defined on $\mathbb{S}_1$ given by $f(x) = \sum_{k < \left\lfloor\frac{n}{2}\right\rfloor} c_k e^{ikx}$, then it is possible to interpolate $f$ by computing coefficients $c_k = \frac{\langle f, e^{ikx} \rangle}{\left\lVert{e^{ikx}}\right\rVert^2}$. Because there are $n$ basis functions and $n$ points, this method produces the unique interpolant for the Fourier basis at the trapezoidal point set for any function $f$. On $\mathbb{HT}$, these properties still hold except that we take the basis to be the set of cosines and the collocation points must respect the reflective symmetry of the domain.

Traditionally, the Chebyshev points are derived by reference to the roots or critical points of the Chebyshev polynomials. But the torus analogy suffices to show that they offer optimal integration accuracy with respect to the Chebyshev weight, as well as to explain perhaps the most crucial property of the Chebyshev nodes. Polynomial interpolation at any of these node sets is equivalent to computing a cosine transform over equispaced points and in particular is amenable to the Fast Fourier Transform. This allows for the mostly unique ability of Chebyshev algorithms to form interpolants in $O(n \log n)$ time. This also has implications for the truncation error of these interpolants as Fourier series determined with equispaced points are well known\cite{bari1964treatise} to have minimally decaying Lebesgue constants of order $\log n$.

\subsection{Discrete Cosine Transform} \label{sec:DCT}
\begin{figure}

\centering
\includegraphics[width=.95\linewidth]{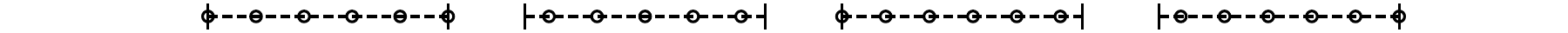}
\caption{From left to right: example point distributions for DCT I, II, $\mathrm{V}_-$, and $\mathrm{V}_+$.}\label{p_types}
\begin{tabular}{|c|c|c|}
\hline 
Chebyshev Point Set & End Points & DCT type \\
\hline
$\cos\left(\pi \dfrac{i}{n - 1}\right)$ for $i = 0\dots n-1$ & $\pm 1$ & $\mathrm{I}$ \\ \hline
$\cos\left(\pi \dfrac{2i + 1}{2n}\right)$ for $i = 0\dots n-1$ & $\emptyset$ & $\mathrm{II}$ \\ \hline
$\cos\left(\pi \dfrac{2i}{2n - 1}\right)$ for $i = 0\dots n-1$ & $1$ & $\mathrm{V}_-$ \\ \hline
$\cos\left(\pi \dfrac{2i + 1}{2n - 1}\right)$ for $i = 0 \dots n - 1$ & $-1$ & $\mathrm{V}_+$ \\ \hline 
\end{tabular}
\end{figure}

As a preliminary to describing the algorithms used here, we will first discuss the 4 types of one dimensional Chebyshev nodes used for Bravais lattices from an algorithmic perspective. All of these points have the same point density, but are distinguished whether the point sets include $\pm 1$. Each of these point sets allow for finding interpolants with a DCT type I, II, or V, depending on the boundary conditions. The information is summarized in~\cref{p_types}. The type I, II, and V Discrete Cosine Transforms compute the following sums over a set of uniform sampling points:
\begin{align}
\mathrm{DCT}[f, x](k) &= a_k \sum_0^{n - 1} c_i f(x_i) \cos(k x_i) \\
c_i &= \left\{ \begin{array}{ll} \frac{1}{2} & x_i = 0, \pi \\ & \\
1  & x_i \neq 0, \pi \end{array} \right. \notag \\
a_k &= \left\{ \begin{array}{ll} \frac{1}{\pi n} & k = 0; \, \mathrm{or} \, -1, 1 \in x \, \mathrm{and} \, k=n-1\\ & \\
\frac{2}{\pi n} & \mathrm{otherwise} \end{array} \right.
\end{align}
This is equivalent to computing a series of inner products over the chosen point set, with the exception of the $k = n-1$ mode of the DCT I. For this case there is an additional normalizing factor, because the type I point set is only accurate up to $k = 2n - 3$. The actual computations are performed through an adaptation of the Fast Fourier Transform. When used with Chebyshev nodes, the $x_i$ are taken to be the equally spaced points obtained by the $\gamma^{-1}$ mapping to $\mathbb{HT}$. While equispaced points on $\mathbb{S}_1$ make up infinite one parameter families, the reflective symmetry of $\mathbf{D}_1$ allows for only 4 distinct distributions for a given point count $n$. Of these sets, only the first two are frequently encountered in numerical methods. The set containing $\pm 1$ is referred to as the Chebyshev points of the second kind and the set without endpoints are the Chebyshev points of the first kind.

While the sums computed by the type I, II, and V DCTs are conceptually similar, the algorithmic details are markedly different. Because of this, FFT code must implement each of them separately. In the case of the type V DCT, although it can be implemented efficiently in principle, most commonly used FFT packages\cite{FFTW.jl-2005} do not implement it. For this reason, we recommend avoiding the asymmetric Chebyshev lattices when possible.

\subsection{Chebyshev Polynomials in Higher Dimensions} \label{sec:cheb_in_high}
In higher dimensions, we generalize the Chebyshev mapping through the tensor product. For each dimension $x_i$, we have $\theta_i = \cos^{-1}(x_i)$.
We will consider the standard $d$-dimensional torus $S_1^d$ as the periodic interval $[-\pi, \pi]^d$. We will often consider the full torus for the sake of expedience, but in fact the $\cos^{-1}(x)$ mapping does not map to the full torus, but to the semi-torus $[0, \pi]^d$. The mapping from $f$ to $\tilde{f}$ defines a function over the full domain through the addition of symmetry across the planes $x_i = 0$. This additional symmetry has implications in how best to represent functions on the torus and what sets of interpolation points respect the symmetry of the domain.\\

For $S_1^d$, the natural basis for analytic functions is $B_{\mathbf{k}}(\boldsymbol{\theta}) = e^{i\langle \mathbf{k}, \mathbf{\theta} \rangle}$ for $\mathbf{k} \in \mathbb{Z}^d$. For the Chebyshev basis, we consider a subset obtained by symmetrizing the basis functions. We have $C_{\mathbf{k}}(\boldsymbol{\theta}) = \prod^d_{i = 1} \cos(k_i \theta_i)$. The set $C$ is related to $B$ by:
\begin{equation}
 C_{\mathbf{k}}(\boldsymbol{\theta}) = \sum_{\boldsymbol{\sigma}} B_{\boldsymbol{\sigma} \mathbf{k}}(\boldsymbol{\theta}) : \boldsymbol{\sigma} \in \{-1, 1\}^d
\end{equation}
This basis can be difficult to work with directly, so it is often expedient to consider $C$ merely as a subset of $B$. \\

One complication in higher dimensions is the notion of polynomial degree. In one dimension, the degree of a polynomial is the exponent of the leading coefficient. For this reason, all polynomial bases are identical in span for the purposes of one dimensional interpolation. In higher dimensions, monomials are written as a product of $d$ factors $x_i^{n_i}$. Our notion of degree is a partial ordering of the set of degree vectors $\mathbf{n} \in \mathbb{Z}_+^{d}$. Commonly used notions are the maximum degree $\lVert\mathbf{n}\rVert_\infty = \max(\mathbf{n})$ and the total degree $\lVert\mathbf{n}\rVert_1 = \sum n_i$. The total degree deserves some special acknowledgment, polynomials do not change in total degree under linear transformation so there is some natural reason to prefer this notion of degree.

For analytic functions on the $d$-cube, the degree we will prefer is the Euclidean degree $\lVert\mathbf{n}\rVert_2 = \sqrt{\sum_i \mathbf{n}_i^2}$. A recent result by Trefethen~\cite{trefEucDegree} suggests that this provides the most consistent bound for the truncation error of an interpolant. His results involve a slight change in the typical notion of analyticity used for functions defined on the cube, but we will attempt to justify this choice in the numerical results.

\subsection{Euclidean Degree Efficiency}
Consider a set of collocation points $S$ in $[-1, 1]^n$ along with a set of Chebyshev polynomials $T$ such that there for any $f : [-1, 1]^n \rightarrow \mathbb{R}$, there is a unique polynomial $p$ in the span of $T$ satisfying $f(x_i) = p(x_i)$ for $x_i \in S$. Because we are concerned with finding points suitable for a Euclidean degree basis, we will define a notion of basis efficiency in the following way:
\begin{equation}
E_{S} = \min_{p_{\mathbf{k}} \notin T}\frac{ V_n\left(\lVert \mathbf{k} \rVert_2 \right)}{2^n |S|} \label{eq:gen_eff}
\end{equation}
Where $\mathbf{k}$ is a multi-index for the Chebyshev polynomial degrees in each direction (as defined in \cref{sec:cheb_in_high}) and $V_n$ is the volume of an $n$-ball as a function of the radius. Intuitively, this represents the ratio of the largest Euclidean degree basis contained in $S$ (approximated by a hemi-spherical volume) to the total number of points in $S$. If the truncation error of a set of points is dependent on the Euclidean degree of the basis, then the the efficiency of a node set $E_S$ represents the smallest node set relative to $|S|$ that could obtain a similar truncation error. Because we will go on to construct more efficient node sets than the tensor product Chebyshev nodes, it is worth determining how much it is possible to save.

A tensor product Chebyshev node set contains $(n+1)^3$ nodes, where $n$ is the maximal degree norm of the tensor product basis. The lowest Euclidean degree excluded from this basis happens to be $n + 1$, so we have:
\begin{align}
E^n_{C} &= \frac{V_n(n + 1)}{2^n (n + 1)^n} \label{eq:cart_eff} \\
E^2_{C} &= \frac{\pi}{4} \approx .785  \label{eq:2_eff}\\
E^3_{C} &= \frac{\pi}{6} \approx .524 \label{eq:3_eff}
\end{align}
The efficiency of the tensor product Chebyshev points decreases quickly with dimension, but for 2 and 3 dimensions we can reduce the point count of the Chebyshev points by about 20\% and 50\% respectively.
\subsection{Lattices and Aliasing}
\begin{figure}[H]
\centering
\includegraphics[width=.4\textwidth]{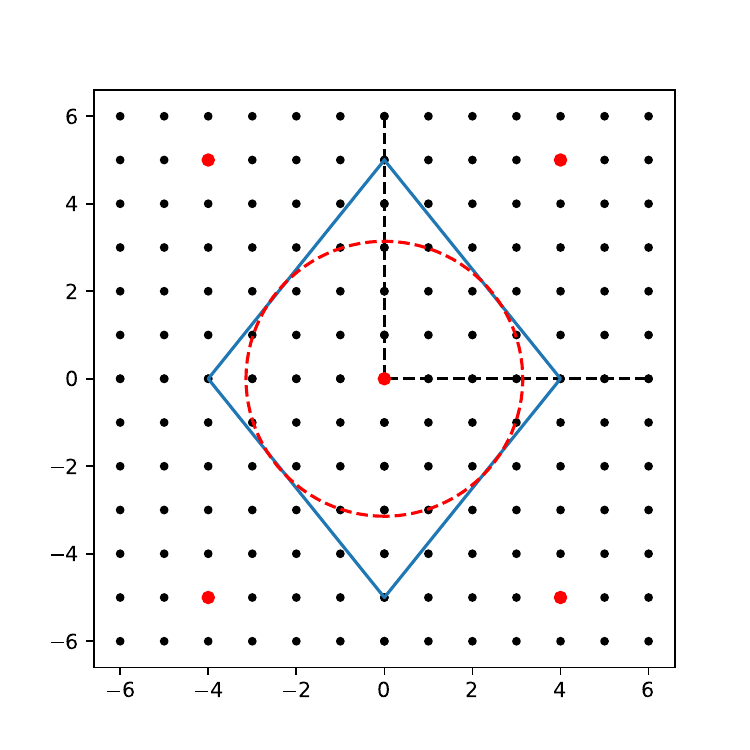}
\includegraphics[width=.4\textwidth]{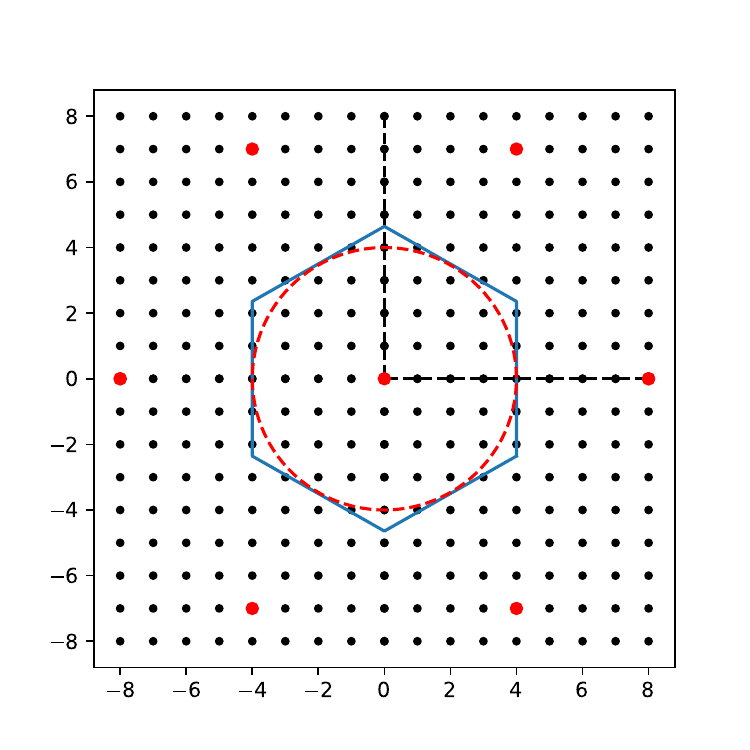}
\caption{Reciprocal lattices to the Padua points and Hex points respectively. Note the circle representing the Euclidean degree and the red dots representing aliasing frequencies.} \label{fig:recip}
\end{figure}
\begin{definition}[Bravais Lattice] \label{bravais-conditions}
A set of points obtained from the orbit of a single point acted on by a translation group. In $d$-dimensions the group can be generated by $d$ basis vectors. The points of a Bravais lattice can be given as:
\begin{equation}
\boldsymbol\theta = P \mathbf{z} + \boldsymbol{\theta}_0
\end{equation}
for $\mathbf{z} \in \mathbb{Z}^d$ and $P$ a $d \times d$ matrix with columns given by linearly independent lattice basis vectors. If $\theta_1$ and $\theta_2$ are elements of a Bravais lattice, then we may refer to $\boldsymbol\theta_1 - \boldsymbol\theta_2$ as a separation vector of the lattice. The set of separation vectors between points in a lattice is closed under addition and subtraction. For any lattice element $\boldsymbol\theta$ and separation vector $\mathbf{v}$, $\boldsymbol\theta + \mathbf{v}$ is also a lattice element.
\end{definition}

Given a Bravais Lattice with basis vectors given by the columns of $P$, we define the reciprocal lattice in terms of $Q$ satisfying $PQ^T = 2 \pi I$. This lattice has the property that for any elements $\mathbf{k} = Q \mathbf{z}_1$ and $\mathbf{\theta} = P \mathbf{z}_2 + \mathbf{\theta}_0$, $\mathbf{z}_1, \mathbf{z_2} \in \mathbb{Z}^d$. We have:
\begin{equation}
e^{i\langle \mathbf{k}, \boldsymbol{\theta} \rangle} = e^{i \langle \mathbf{k}, \boldsymbol{\theta}_0 \rangle}
\end{equation}
This property is a generalization of aliasing in one-dimension. Elements of the reciprocal lattice are represented in the basis space as functions which are identically constant when evaluated at each lattice point. This property is fundamentally related to a crucial numerical property of Bravais lattices---Gauss-Legendre integration accuracy.

Consider the sum:
\begin{equation}
S = \sum_{p \in P} e^{i\langle k, p \rangle}
\end{equation}
Because of the translational symmetry of this lattice we have that for $p_0 \in P$:
\begin{align}
S &= \sum_{p \in P} e^{i \langle k, p_0 + p \rangle} \notag \\
S &= e^{i\langle k, p_0 \rangle}\sum_{p \in P} e^{i \langle k, p \rangle} \notag \\
S &= e^{i\langle k, p_0 \rangle} S
\end{align}
It follows as an immediate consequence that if $\mathbf{k}$ is not an element of the reciprocal lattice, then $S = 0$. The one dimensional version of this proof is commonly used for proving spectral accuracy for trapezoidal integration in periodic functions. Another consequence of this property is that the trapezoidal integral stencil can accurately calculate inner products of two polynomials up to degree $n - 1$. This is provably the best case for an orthogonal basis set because it implies that the interpolating polynomial for $f$ over $P$ can be determined by computing a series of inner products. If it were possible to integrate even one degree more accurately, then this method of computing inner products would allow for fitting an interpolant with more functions than the number of points in $P$.

So far, this section has discussed lattices in $\mathbb{R}^d$. In $\mathbb{S}_1^d$, lattices behave analogously with two differences. Because of the periodicity of the torus, the basis set of complex exponentials is discrete. For a torus with fundamental domain $[-\pi, \pi]^d$, $e^{ikx}$ is single-valued only for $k \in \mathbb{Z}^d$. Separation vectors $\mathbf{u}$ in a finite lattice are also constrained such that each $\mathbf{u}_i = \frac{a}{b} \pi$ for $a, b \in \mathbb{Z}$. Choosing lattice vectors is equivalent to choosing aliasing frequencies, so it can be helpful to design real lattices by choosing an appropriate reciprocal lattice. Bravais lattices on $\mathbb{HT}^d$ have an additional symmetry requirement. Reflective symmetry in the torus means that if two elements in a Bravais lattice are separated by a vector $\mathbf{u}$, then any vector obtained by inverting signs of $\mathbf{u}$ is also a Lattice vector.

These constraints dramatically reduce the search space for good lattices. There are a finite number of viable configurations in a given dimension and each one only has $d$ degrees of freedom with domain $\mathbb{N}$. In 2 dimensions, lattices can either be tensor product lattices with basis vectors $[\pi/{n_1}, 0], [0, \pi/{n_2}]$ or Padua-type lattices with basis vectors $[{\pm \pi/{n_1}, \pi/n_2}]$. Three dimensional lattices similarly have only a fundamentally distinct configurations. There are tensor product, body-centered cubic, face-centered cubic, and the tensor product of a one dimensional lattice with a Padua-type lattice. Of the three dimensional lattices, only the BCC and FCC are of particular theoretical interest.

\subsection{Special Properties of Chebyshev Lattices} \label{sec:grid-proof}

Bravais lattices on $\mathbb{HT}^n$ have additional symmetries that are useful for developing algorithms and finding optimal lattices for interpolation. The reflective symmetry of $\mathbb{HT}^n$ implies that if $(x_1, \dots, x_n)$ is a lattice vector, then every vector $(\pm x_1, \dots, \pm x_n)$ obtained by choosing an arbitrary sign for each element is also a lattice vector. From this we will establish a few important properties of Chebyshev lattices.

We define a Cartesian lattice on $\mathbb{HT}^n$ to be the tensor product of axis aligned 1 dimensional lattices. The basis for such lattices can be expressed as $a_i \hat{\mathbf{x}}_i$, where $a_i = 2\pi/n_i$ is the minimal separation between two lattice points along the $i$th axis. Aside from the choice of $a_i$ 1 dimensional lattices of this kind are distinguished by whether or not they contain the points 0 or $\pi$. If $n_i$ is odd, the lattice must contain one of 0 and $\pi$. If $n_i$ is even, it may contain both or neither.
\begin{theorem} \label{thm:cheb_bra}
The following statements are true:
\begin{enumerate}
\item Every Bravais lattice on $\mathbb{HT}^n$ can be written as a union of $2^m$ Cartesian lattices where $m \in [0, 1, \dots, n-1]$. 
\item Every Bravais lattice on $\mathbb{HT}^n$ is a subset of a Cartesian lattice.
\end{enumerate}
\end{theorem}
\begin{proof}
Let L be a Bravais lattice on $\mathbb{HT}^n$. For an axis $i$ and an element $\ell \in L \cap [0, \pi]^n$, the separation between $\ell$ and its reflections along the planes with normal $\hat{\mathbf{x}}_i$ at $x_i = 0, \pi$ are $2\ell_i$ and $2\pi - 2\ell_i$ respectively. For each axis $i$ let $a_i$ be the smallest non-zero positive value obtained in this way. Then $a_i \hat{\mathbf{x}}_i$ is an axis align lattice vector. Let $\mathbf{v}_i$ for $i \in [1, \dots, n]$ be a basis for $L$'s lattice vectors. Then by adding linear, integer combinations of vectors $a_i \hat{\mathbf{x}}_i$ to a given $\mathbf{v}_i$ we can produce a new vector $\mathbf{w}_i$ such that each $w_{i, j}$ satisfies $0 \leq w_{i, j} < a_j$. The combined set of vectors $\mathbf{w}$ and $a_i \mathbf{x}_i$ are a generating set for $L$'s lattice vectors. 

Using the reflective symmetry of $\mathbb{HT}^n$, we can reflect a vector $w_i$ about each axis except axis $j$ to produce $\mathbf{w}_i^*$. We have that $\mathbf{w}_i + \mathbf{w}_i^* = 2w_{i, j} \hat{\mathbf{x}}_j$. Because $0 \leq w_{i, j} < a_j$, we have that $|2 w_{i, j} - a_j| < a_j$. This implies that $w_{i,j} = a_j/2$, or $0$, otherwise $\mathbf{w}_i + \mathbf{w}_i^*$ can be used to find a lattice point closer to a $j$-axis reflecting plane than $a_j/2$ which contradicts our assumptions. It follows that every entry of $\mathbf{w}_i$ is $0$ or $a_j/2$ for any basis constructed in this manner.

From this, we see that each non-Cartesian Bravais lattice $L$ on $\mathbb{HT}^n$ has a largest Cartesian subset with minimal separation vectors $a_i \hat{\mathbf{x}}_i$ and the basis for $L$ can be written as a combination of $n$ vectors consisting of some combination of the axis-aligned vectors $a_i \hat{\mathbf{x}}_i$ and vectors $\mathbf{w}_i$ with at least 2 non-zero components $w_{i,j} = a_j/2$. From this, we see immediately that $L$ is a subset of the Cartesian lattice with basis vectors $\frac{a_i}{2} \hat{\mathbf{x}}_i$. 

Let $L^+$ denote the Cartesian super lattice defined in this way, and let $L^- + \ell_i$ denote the maximal Cartesian sublattice of $L$ containing $\ell_i$. It is a trivially-shown property of lattices that if $L^- + \ell_i \neq L^- + \ell_j$, then $L^- + \ell_i \cap L^- + \ell_j = \emptyset$. If $L$ is a proper subset of $L^+$, then there is some $j$ such that $\ell_i + \frac{a_j}{2} \hat{\mathbf{x}}_j \notin L$ for any $\ell_i$. The set of points $\tilde{L}$ obtained from the disjoint union $L \cup L + \frac{a_j}{2} \mathbf{x}_j$ has twice as many points as $L$. And because $a_j \hat{\mathbf{x}}_j$ is a lattice vector of $L$, for every point $\ell_i \in \tilde{L}$,  $\ell_i + \frac{a_j}{2} \hat{\mathbf{x}}_j \in \tilde{L}$, i.e. $\tilde{L}$ is a lattice. We can also see that it is a subset of $L^+$. This procedure can be repeated as long as $L$ is a proper subset of $L^+$. This doubling process implies that any such lattice $L$ must have a point count equal to $|L^+|/2^k$ for some $k$. By definition $L^-$ has a point count of $|L^-| = |L^+| / 2^d$. Therefore $|L| = 2^{m}|L^-|$ for $0 \leq m < d$. This implies the first statement of the hypothesis, because $L$ is a union of $2^m$ translated copies of $L^-$. 
\end{proof}

\subsection{Composite Lattices} \label{sec:comp_lats}
\begin{figure}[H]
    \centering
    \includegraphics[width=.4\textwidth]{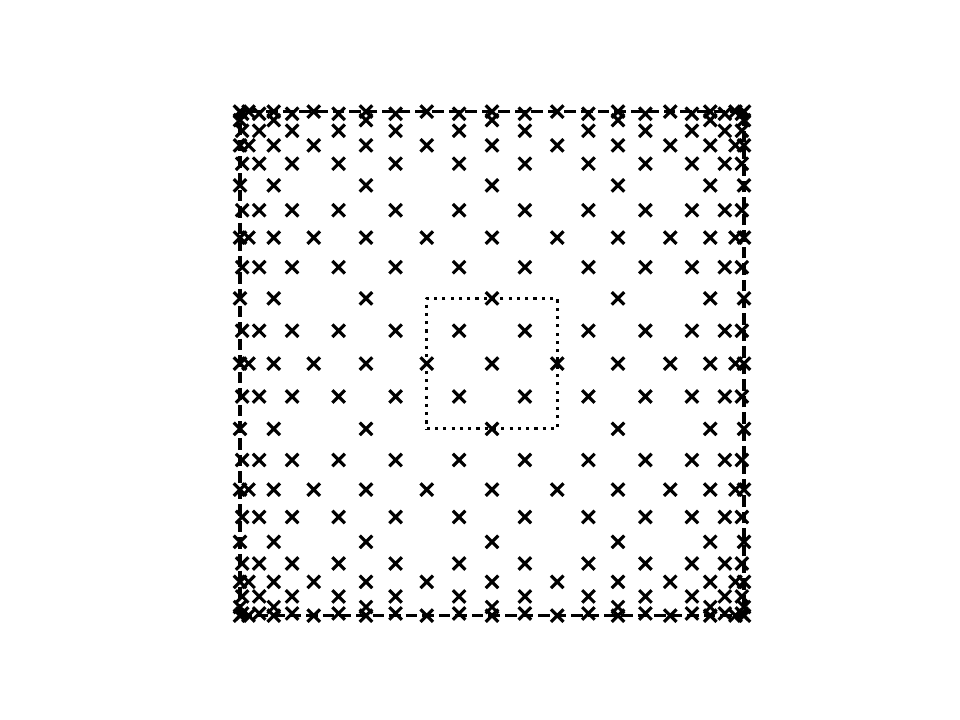}
    \includegraphics[width=.3\textwidth]{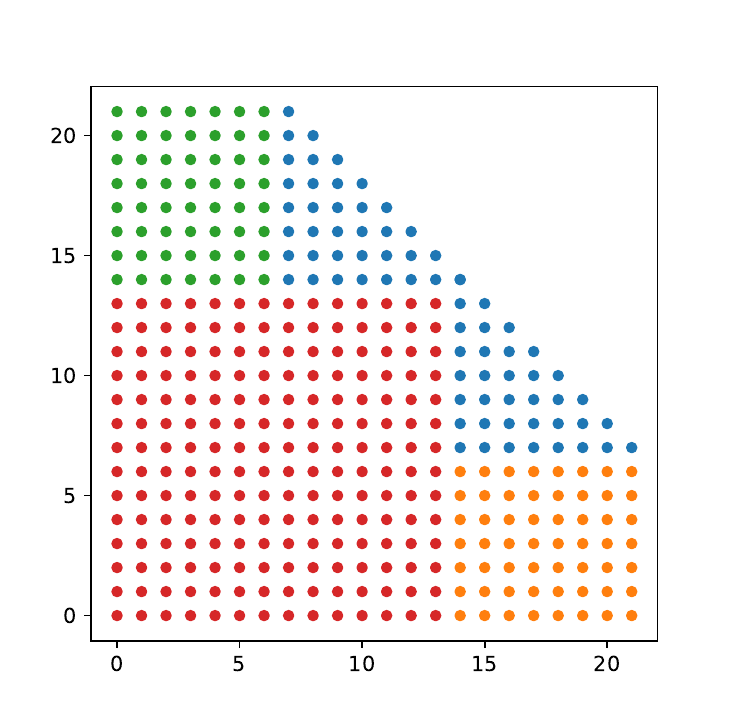}
    \caption{Left: Example of a seven point composite lattice. The dotted box contains a single unit cell. Right: An example index set of the basis for a seven-point composite lattice. Coefficients whose indices are plotted with matching colors are stored in a single contiguous storage array.}
    \label{fig:enter-label}
\end{figure}
Bravais lattices are fully characterized through the orbit of one point under a set of translations. Composite lattices extend this notion by instead considering a set of points within an asymmetric unit cell.
\begin{definition}[Composite Lattice]
    A $d$-dimensional composite lattice on the reflected torus consists of a set of translations generated by $d$ vectors satisfying the same conditions as in \cref{bravais-conditions} along with a set of $n$ points that are not separated by lattice vectors. The composite lattice is the union of the lattices produced by the orbit of these $n$ points. The unit cell of a composite lattice consists of the nearest neighbor cell of a point at the intersection of $d$ hyperplanes of symmetry with respect to the Bravais lattice generated by the orbit of that point under the translational symmetry of the composite lattice.
\end{definition}
Bravais lattices are limited in their achievable Euclidean degree efficiency, with only a small subset of lattices having useful interpolation properties in a given dimension. Composite lattices allow for lattices which approach theoretically optimal efficiency without sacrificing the $O(n \log n)$ efficiency of a Bravais lattice. The primary source of difficulty for composite lattices is the polynomial basis for the lattice. As with Bravais lattices, the basis is determined by geometric constraints which dictate the unisolvent basis of the lattice. This is determined through consideration of Brillouin zones:
\begin{definition}[Brillouin Zone]
    Given a Bravais lattice $L$, consider the reciprocal lattice. Given a point $p_i$ in the basis space for $\mathbb{S}^d_1$ consider the set of points produced by adding some element $\ell_j$ of the reciprocal. This set is partially ordered with respect to $\lVert p_i + \ell_j \rVert_2$. The set of all points $p_i + \ell_j$ which can be the $n$th elements of such a set make up the $n$th Brillouin zone of $L$. When multiple elements of a set have the same $L_2$ norm, these points are the boundaries of multiple Brillouin zones. The first Brillouin zone of a lattice is the optimal basis for that lattice in Euclidean degree, the union of the first $n$ Brillouin zones are the optimal basis for an $n$-point composite lattice. For expedience we will refer to the union of the first $n$ Brillouin zones of a lattice as Brillouin unions.
\end{definition}
The above definition considers $\mathbb{S}^d_1$ for expedience, it is trickier to define the Brillouin zone for $\mathbb{HT}^d$, but the outcome is identical to the $\mathbb{S}_1^d$ Brillouin zone restricted to the first quadrant of the dual space. Higher order Brillouin unions for a given lattice tend toward higher Euclidean degree basis efficiency, but this tendency is not monotonic, and Brillouin unions are not generally convex. This causes some difficulty in finding good node sets, but one non-trivial example in 2d is that the 7th Brillouin union of an equal aspect ratio Cartesian lattice is an irregular octagon.

One useful corollary of \cref{thm:cheb_bra} is that every Bravais lattice is a composite lattice with at most $2^{d-1}$ points in the unit cell for some Cartesian lattice. It follows that all composite lattices on $\mathbb{HT}^n$ can be taken to have a Cartesian unit cell.
\section{Optimal Lattices in Two and Three Dimensions}
This section will survey lattices of interest in two and three dimensions. Optimality in the Euclidean degree will be judged according to the Euclidean degree efficiency defined in \cref{eq:gen_eff}. Because this measure is a ratio of the largest ball contained in the basis to the basis size, we can determine a general principle for finding the best lattice. The radius of the largest ball in the basis for a Bravais lattice is half the distance between nearest neighbors of the reciprocal lattice. The volume of the nearest neighbor cell is the volume of a ball combined with the empty space nearest it. These definitions are equivalent to the density of a packing of equal spheres, making the search for optimal lattices equivalent to the close packing problem in the dual space to $\mathbb{HT}^n$. Due to symmetry and periodicity restrictions, it is not always possible to translate the close packing lattice in $\mathbb{R}^n$ to $\mathbb{HT}^n$. 
Another point to consider for lattices is their representation of boundary conditions. The value of a function on the boundaries of a cubic domain are given by the projection of the basis onto a $d-1$ dimensional coordinate hyperplane. The efficiency of the basis in this sense may be a more pertinent consideration for solving boundary value problems.
\subsection{2d} \label{sec:hex}
\begin{figure}[H]
\centering
\includegraphics[width=\linewidth]{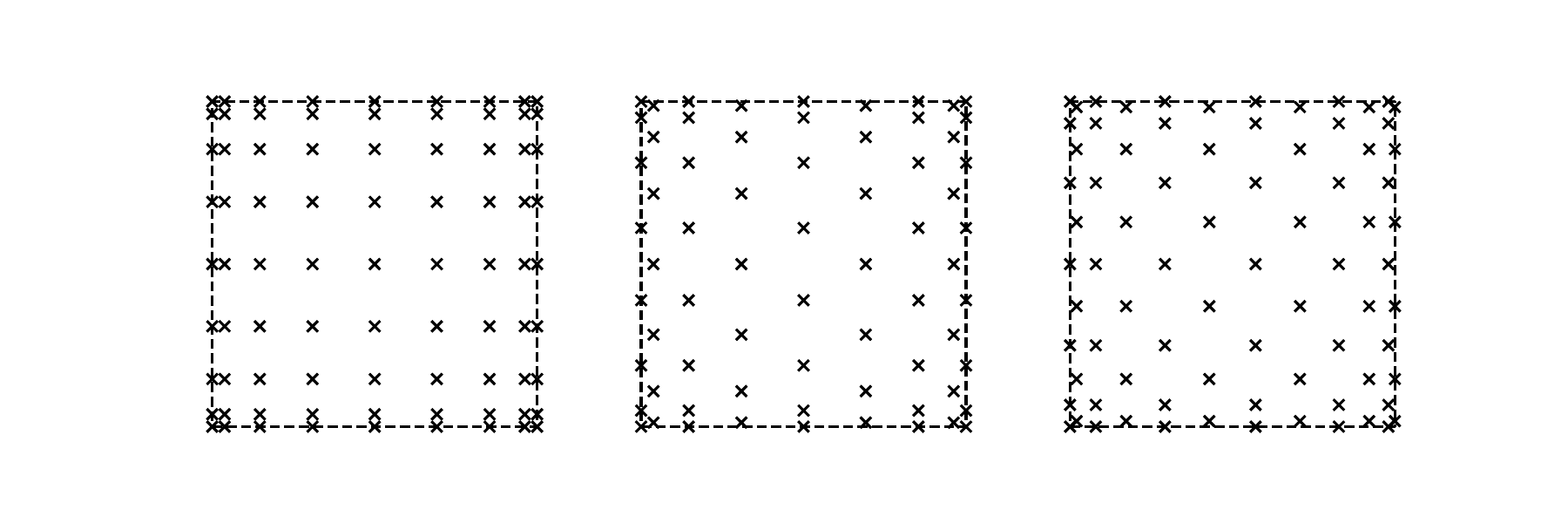}
\caption{From left to right: Cartesian, Hexagonal, and Padua lattices of Euclidean degree 8. The point counts are 81, 68, and 78 respectively.}
\label{fig:hex-lat}
\end{figure}
In 2d, the close packing lattice is the self-dual hexagonal lattice. This has lattice vectors $[1, 0]^T, \frac{1}{2}[1, \sqrt{3}]^T$. Because the reciprocal lattice vectors must be integers, it is not possible make an exact hexagonal lattice. A close approximation can be made with vectors $[8n,0]^T [4n, 7n]^T$. The first lattice vector is slightly shorter than the second, so a circle inscribed in the basis set would have a radius of $4n$. The calculated basis efficiency would be:
\begin{align}
E_{hex} = \frac{\pi 4^2}{\left| \begin{matrix} 8 & 0 \\ 4 & 7 \end{matrix} \right|} \notag \\
E_{hex} = \frac{2 \pi}{7} \approx .898
\end{align}
Comparing this to the 2d Cartesian basis efficiency (\cref{eq:2_eff}), we see that the approximated hex lattice attains the same Euclidean degree with $\frac{7}{8}$ times as many points as a tensor product lattice. At high resolutions, it would be possible to get slightly better approximations; but this would make for an improvement of less than $.01$ in the efficiency factor. 

In 2 dimensions, it is possible to construct a Bravais lattice which is unisolvent for total degree polynomials. This was discovered in 2008 at the university of Padua.\cite{PadOriginal}\cite{pad_interp_generate}\cite{pad_interp_xu_de_marchi} Although the relation to toroidal lattices was not made until sometime after, the Padua points come from the lattice associated with the reciprocal vector $[n, n+1]^T$ and its reflections across the coordinate axes. The $\ell_1$ and $\ell_\infty$ unit balls are geometrically identical in 2 dimensions, so the Padua points are no more efficient than a tensor product lattice in Euclidean degree. See \cref{fig:hex-lat} for a comparison of the lattices, and \cref{fig:recip} for a comparison of the Hex and Padua bases.
Unlike the Hexagonal and Cartesian lattices, the Padua points are more efficient at representations along the boundaries by a factor of $\sqrt{2}$. This makes the Padua points optimal for some boundary value problems.

\subsection{3d} \label{sec:b-eff}
In 3d, there are infinitely many distinct close packings, but only one is a Bravais lattice. This is commonly referred to as the face-centered cubic (FCC) lattice. Fortunately, this lattice can be exactly expressed with integer vectors: $[1, 1, 0]^T, [1, 0, 1]^T, [0, 1, 1]^T$. It is not self-dual; a basis with FCC symmetry is dual the to body-centered cubic (BCC) lattice with vectors given by some multiples of $[-1, 1, 1]^T, [1, -1, 1]^T, [1, 1, -1]^T$. Both of these lattices are more efficient than the Cartesian lattice, so we will calculate the efficiency factors for both:
\begin{align}
E_{BCC} &= \frac{4\pi}{3\sqrt{2}^3 \left|\begin{matrix} 0 & 1 & 1 \\ 1 & 0 & 1 \\ 1 & 1 & 0 \end{matrix} \right|} = \frac{\pi}{3 \sqrt{2}} \approx .740 \\
E_{FCC} &= \frac{\pi \sqrt{3}^3}{6 \left| \begin{matrix} -1 & 1 & 1 \\ 1 & -1 & 1 \\ 1 & 1 & -1 \end{matrix} \right|} = \frac{\pi \sqrt{3}}{8} \approx .680
\end{align}
Compared to the Cartesian lattice (see \cref{eq:3_eff}), the BCC attains the same error with $\frac{1}{\sqrt{2}} \approx .71$ of the points. The FCC has a worse efficiency, saving points by a factor of $\frac{4}{3\sqrt{3}} \approx .77$. The BCC lattice also happens to be optimal for total degree, as observed in \cite{total_degree_cube_demarchi}. However, it is possible to save about $\frac{3}{4}$ of the point count using more efficient composite lattices.

The FCC has a slight advantage in its ability to resolve functions representing boundary conditions. Given a set of basis functions in $d$ dimensions, the basis for functions restricted to one boundary can be thought of as the projection of the basis indices along one axis. In this way, it is possible to consider the Euclidean degree efficiency of a lattice in representing boundary conditions separate from its efficiency in its volume. The FCC has a higher order basis in Euclidean degree along its boundaries than in its interior by a factor of $\frac{2}{\sqrt{3}}$ making its boundary efficiency $~.785$. This makes it slightly more efficient than the BCC for boundary representations.

\subsection{Higher Dimensions}
This process can be readily generalized to higher dimensions. The ratio of efficiency from the Cartesian lattice to the closest packing admitted by the Chebyshev torus tends to increase with dimension. Using the same methods as previously, one can show that the D4 lattice is exactly twice as efficient as the 4 dimensional Cartesian lattice. At the same time, the closest packing lattice itself becomes exponentially less efficient than the theoretical limit. As well, some close packing lattices do not respect the symmetry of the torus. The 24-dimensional leech lattice for example has no hyperplane symmetry, and so there is no associated polynomial-collocation method.

\subsection{Composite Lattices}
In principle, composite lattices can be made to be arbitrarily efficient. However, this sacrifices some computational efficiency, as the lack of translational symmetry at the level of the unit cell means the Vandermonde matrix cannot generally be compressed at that level. For a $p$ point composite lattice with $n$ points total, the computational cost of interpolation is $O(p^2n + n\log n)$. Due to the $p^2$ complexity in interpolating with composite lattices, composite lattices with large $p$ may be impractical at low resolution.

At present, the only non-trivial composite lattice with a convex basis is the $7$-point 2d lattice shown in \cref{fig:enter-label}. This lattice has a basis efficiency of about $.898$, exactly the same as the approximate hexagonal lattice (\cref{sec:hex}). This lattice is included in the numerical results mostly as a proof of concept for composite lattices in general. One concrete benefit to this lattice is its improved boundary efficiency over the hexagonal lattice. (This concept is described briefly in \cref{sec:b-eff}).

\section{Algorithms}

Previous papers on the Padua points and BCC lattice have proposed algorithms~\cite{PadAlgo2, fcc_integrate_xu} which can compute coefficients in $O(n \log n)$ time or could be adapted to achieve this. However, these methods involve computing DCTs over larger sets than the points used. This means that so long as evaluating the function to be interpolated is not a bottleneck, the more efficient point sets require more computation to find coefficients than a simple tensor product lattice. For this reason, we propose two novel methods for generalizing the tensor product DCT to non-tensor product lattices. The intent of these algorithms is to match the computational efficiency of the DCT computing coefficients on a tensor product Chebyshev lattice. This way, a well-implemented interpolation scheme can expect computational performance reflecting the efficiency of the basis. 

Both algorithms we propose use Cartesian sublattices as a basic unit. Both Bravais and composite lattices must have some Cartesian translational symmetry due to the toroidal geometry of the domain. For Bravais lattices, the hyperplane symmetry requirement means that any lattice vector can be added with one of its reflections to produce an axis aligned vector. It follows from this that all Bravais Chebyshev lattices can be written as a union of at most $2^{d - 1}$ Cartesian lattices. Composite lattices do not have any upper bound on the number of sublattices, but must have some Cartesian lattice symmetry. The Padua points, Hex Lattice, and BCC lattice are composed of two sublattices; the FCC lattice is composed of 4; and the Octagonal composite lattice is composed of 7. The majority of computation for the algorithms we propose consists of computing DCTs or FFTs on these sublattices. The rest is additional $O(n)$ processing, similar in computation to one level of an FFT.

\subsection{Interleaved Fast Fourier Transform}
\begin{figure}[H]
\centering
\includegraphics[width=.95\linewidth]{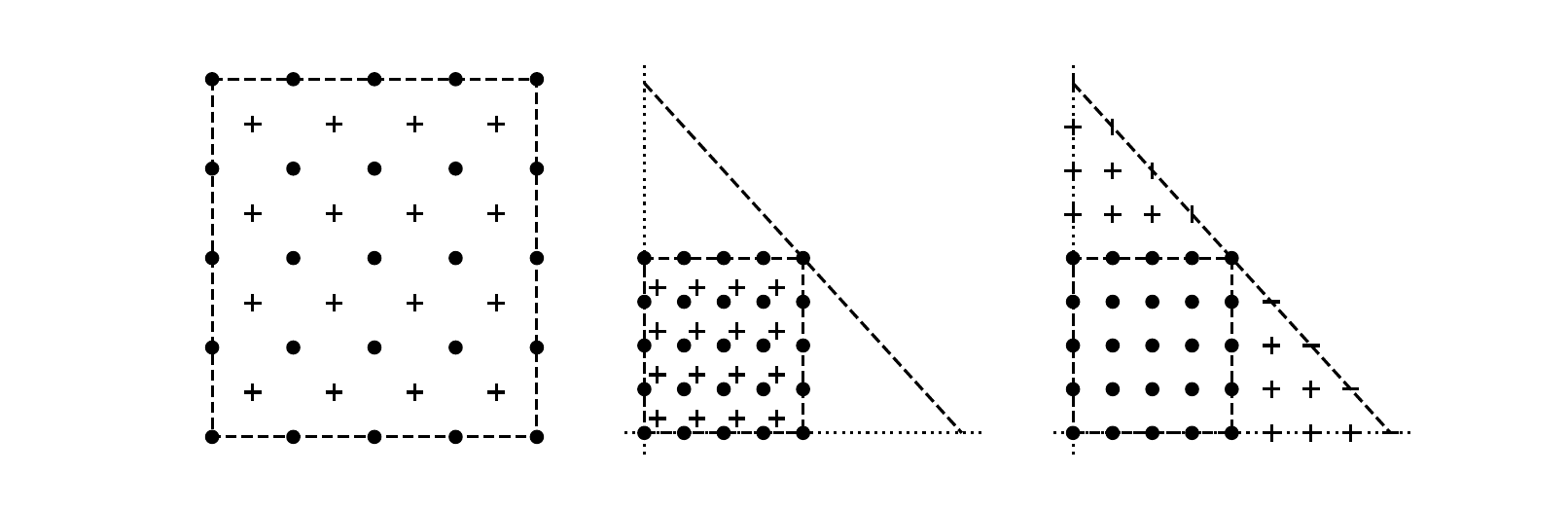}
\caption{A step by step illustration of the IlFFT. Left: Compute the DCT on interleaved lattices, Center: Obtain sets of corresponding coefficients, Right: Use aliasing to obtain coefficients for the optimal Euclidean degree basis for this lattice. Indices marked by lines alias onto different indices with the same Euclidean degree. In this case the basis is taken to contain a linear combination of the two associated polynomials.} \label{pad_dia}
\end{figure}
This algorithm is intended for use when the full node set can be separated into equal lattices with DCT $I$ or $II$ boundary conditions in all directions. It can be applied in cases with type $V$ DCTs if necessary.

First, we divide the collocation points into a set of $m$ disjoint tensor product lattices $L_{1,\dots, m}$ with dimensions $(n_1, \dots, n_d)$. For bravais lattices, the separation between the lattices under the mapping to the torus will be either $\frac{\pi}{n_i}$ or 0 in each direction. The algorithm for determining the set of coefficients $c_{i, \dots, m}$ to interpolate a function $f$ is the following:
\begin{algorithm}
\caption{Interleaved Fast Fourier Transform}
\begin{algorithmic}[1]
\Procedure{ILFFT}{f, L}
\ForAll{lattices $L_i$}
\State $F_i = \mathrm{DCT}(f(L_i))$
\EndFor
\ForAll{indices $\mathbf{k}$ in each $F_i$}
\State $v_{\mathbf{k}} = \left[ F_{1, \mathbf{k}} \dots F_{m, \mathbf{k}} \right]$
\State $M = \mathrm{AliasingMatrix}(\mathbf{k})$
\State $\left[c_{1,\mathbf{k}} \dots c_{m, \mathbf{k}} \right] = M v_{\mathbf{k}}$
\EndFor
\EndProcedure
\end{algorithmic}
\end{algorithm}
A few pieces of this algorithm are left undetermined in the general case. First, the DCT evaluated for each lattice must be type I or II as appropriate for each axis' reflection conditions. Second, the "AliasingMatrix" function produces a matrix $M$ which is dependent on the index $\mathbf{k}$. For Bravais lattices, $M$ is a unitary matrix with entries $\pm \frac{1}{m}$. In all cases, $M$ can takes on a fixed set of values dependent on the structure of the lattice and independent of the resolution. One method to determine the set of aliasing matrices is to directly evaluate sets of mutually aliasing polynomials and find the associated coefficients at the component lattices. This gives the inverse of the aliasing matrix used in the algorithm. This method is especially useful for composite lattices, for which it can be difficult to determine the aliasing matrices analytically.

To evaluate the transform of a function $f$, we first evaluate $f\left(L_j^{\mathbf{x}} \right) = y_j^{\mathbf{x}}$ for each point $\mathbf{x}$ at each lattice $L_j$. Using the appropriate DCTs we can compute a set of coefficient lattices: $c_j = \mathrm{DCT}(f_j)$. Each coefficient lattice $c_j$ will have coefficients $c_j^{\mathbf{k}}$ corresponding to polynomials of the same order. (With the exception that the highest order modes of type $I$ lattices do not have corresponding coefficients in type $II$ lattices). For each polynomial index $\mathbf{k}$ we select the $m$ lowest order polynomial indices which alias onto $\mathbf{k}$, $\mathbf{k}'_{1,\dots, m}$. (Note: the lowest order element of each set $\mathbf{k}'$ is always $\mathbf{k}$). We can construct a set of linear systems:
\begin{equation}
 \left[ \begin{matrix} c_1^{\mathbf{k}} \\ \vdots \\ c_m^{\mathbf{k}} \end{matrix}\right] = M_{\mathbf{k}} \left[ \begin{matrix} c^{\mathbf{\ell}_1} \\ \vdots \\ c^{\mathbf{\ell}_m} \end{matrix}\right]
\end{equation}
For bravais lattices, the matrix $M_{\mathbf{k}}$ is unitary with all elements being $\pm \frac{1}{m}$. The subscript denotes the fact the $M$ is in fact a piecewise constant function of $\mathbf{k}$. The entries $(M_\mathbf{k})_{i,j}$ are given by the value of $c_i^\mathbf{k}$ when evaluated for $T_{\mathbf{\ell}_j}$.

The matrices $M_\mathbf{k}$ have important implications in hypothetical contexts involving composite lattices (mentioned in future work). The DCT is unitary, so the matrices $M_\mathbf{k}$ determine the norm of the Vandermonde matrix.

The total operation count of this algorithm is dominated by $m$ DCTs computed over the lattices $L_i$. The remaining operations consist of matrix multiplications by $m \times m$ matrices for each set of corresponding coefficients. The asymptotic operation count for this algorithm over a set of $mn$ points is then $O(m n \mathrm{log}_2 (n) + m^2n)$. The value of $m$ is fixed by the lattice structure; for large $n$, the bottleneck becomes the initial step of computing an DCT for each $L_i$.

\subsection{Heterogeneous Fast Fourier Transform}
\begin{figure}[H]
\centering
\includegraphics[width=.95\linewidth]{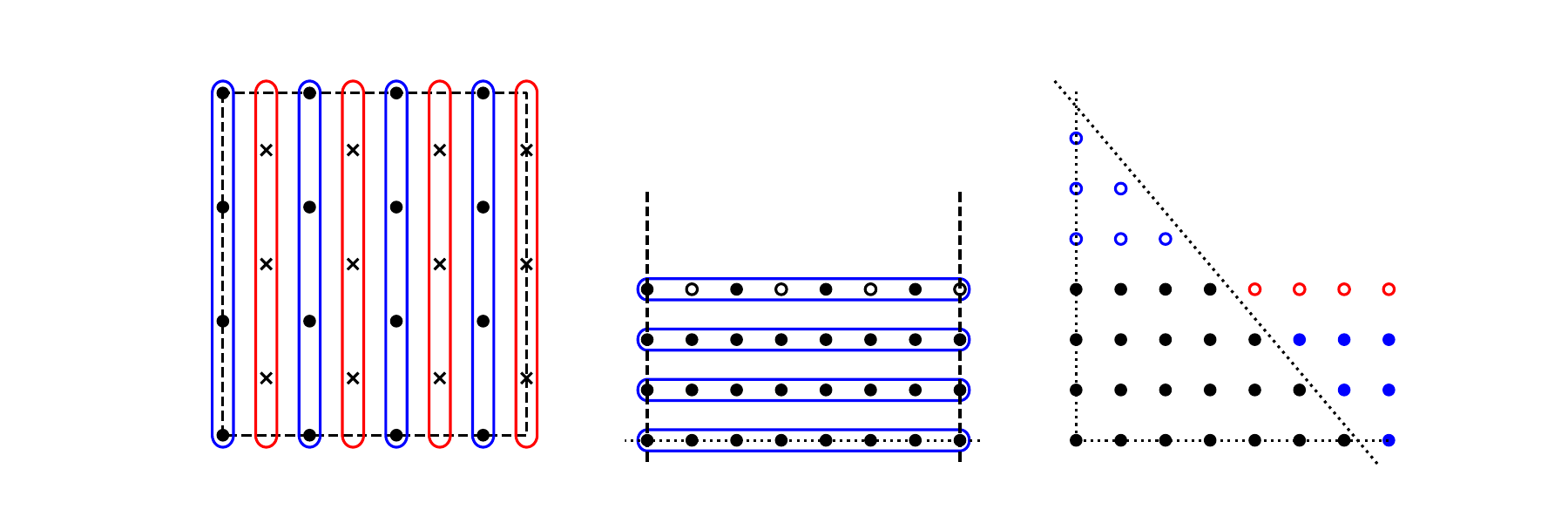}
\caption{A step by step illustration of the HFFT. Left: Compute DCT I and II on alternating columns, Center: Compute DCT I on rows (unfilled circles correspond to coefficients which are always 0), Right: Use aliasing to remap $2n \times n$ coefficients to the total degree basis set. (Red circles correspond to redundant coefficients. Blue dot coefficients outside the basis map to blue circles inside.)} \label{pad_dia}
\end{figure}
The Interleaved FFT is a generic algorithm that can always be applied in principle. However, in some cases the SFFT may require use of a type V DCT. In particular, the Padua points have type V boundary conditions along one axis. As remarked upon in \cref{sec:DCT}, the type V DCT is not commonly implemented in FFT libraries. To compute an ILFFT without a specially implemented DCT it would be necessary to compute a real FFT on an array with reflected data. This is a meaningful time loss, which can be avoided by circumventing the need for a type V DCT.\\
Figure (\ref{pad_dia}) illustrates the steps of this method. Note in the first diagram that transforming the sublattices horizontally would require a type V DCT. To avoid this, we transform the two sublattices using the DCT type I and II vertically first. This algorithm can be generalized to a variety of cases, but for simplicity we will describe the process for a Padua lattice. Without loss of generality, we can assume that the Padua lattice is the union of two Cartesian lattices $L_{1,2}$ of type $\mathrm{I} \times \mathrm{V}_-$ and $\mathrm{II} \times \mathrm{V}_+$. We take $f$ to be a function defined over the domain.
\begin{algorithm}[H]
\caption{Heterogeneous Fast Fourier Transform}
\begin{algorithmic}[1]
    \Procedure{HFFT}{$f$, $L_1$, $L_2$}
    \State $k \leftarrow \mathrm{Total \, Degree}(L_1, L_2)$
    \State $\chi_1 \leftarrow \mathrm{DCT}(f(L_1), \mathrm{Type\, I}, \mathrm{axis}=1)$
    \State $\chi_2 \leftarrow \mathrm{DCT}(f(L_2), \mathrm{Type\, II}, \mathrm{axis}=1)$
    \State $\chi[::2] \leftarrow \chi_1$
    \State $\chi[1::2] \leftarrow \chi_2$
    \State {$\chi \leftarrow \mathrm{DCT}(\chi, \mathrm{Type\, I}, \mathrm{axis}=2)$}
    \ForAll{$\chi [i, j]$}    \Comment{Now we use aliasing math to find the correct index for each coefficient.}
    \If {i + j > k}
    \State $c[k+1 - i, k - j] \leftarrow \chi [i, j]$
    \Else
    \State $c[i, j] \leftarrow \chi[i,j]$
    \EndIf
    \EndFor
    \EndProcedure
\end{algorithmic}
\end{algorithm}
Note that for this algorithm requires that there is some way of ordering transforms to circumvent the DCT V. In most cases, it is more practical to choose lattices that are compatible with the ILFFT. For the specific case of the Padua points, this method is recommended.

It is difficult to precisely determine a flop count for this algorithm, because the DCTs computed depend on the structure of the lattice in a way which may be complicated. However the size of the DCT computed is $n + O(n^{(d - 1)/d})$ at all steps and is followed by an $O(n)$ aliasing correction. This suggests that the algorithm is asymptotically equal to a $d$-dimensional DCT computed over $n$ points.
\subsection{Differentiation}
One of the most notable disadvantages of non-Cartesian lattices is the absence of tensor product symmetry. This is most notable in the construction of differential operators. For Cartesian lattices, differentiation matrices are sparse requiring $O\left(n^{\frac{d + 1}{d}}\right)$ time to compute derivatives directly. For other lattices, this is not the case and direct differentiation requires $O\left(n^2\right)$ time. \\
However, algorithms that use the basis representation to differentiate are still $O(n \log n)$. The most efficient way to compute a derivative in basis space is to backsolve using the sparse recurrence relation:
\begin{equation}
\frac{\mathrm{d}}{\mathrm{d}x} T_n = 2n T_{n - 1} - \frac{n}{n - 2}\frac{\mathrm{d}}{\mathrm{d}x} T_{n - 2}
\end{equation}
The unusual structure of the basis complicates differentiation in this manner somewhat. The method we suggest is to compute a lexical sort which takes the axis of differentiation as the least significant component. The sorting step can be performed once upon initializing a lattice. Because the structure of the basis is known a priori it is possible to use a series of $2d - 1$ stable bucket sorts to produce appropriate orderings for each axis of differentiation.
We store the components in a vector according to this order. Then for each sequence of coefficients whose indices differ only in the least significant index, we can apply the recurrence relation to compute the derivative along an axis. This process takes $O(n)$ flops. In the limit of large $n$ differentiation is dominated by the $O(n \log n)$ cost of transforming a function from its values at the collocation points to its basis coefficients.
\subsection{Integration}
While polynomial integration is most often done using Gauss-Legendre stencils, it is well-known that for broad classes of functions, Clenshaw-Curtis quadrature achieves similar levels of accuracy.\cite{trefGQorCC}\cite{trefExactnessQuadrature} With a more efficient basis, Clenshaw-Curtis quadrature can outperform Gauss-Legendre.
Integrals and integral operators can be computed directly using the Chebyshev basis. We will denote by $I$ the linear functional in the Chebyshev coefficients representing the integral over the unit interval. Given a function $f$ defined on the unit interval and the operator $T$ mapping $f$ to its Chebyshev coefficients, we have:
\begin{align}
\int_{[-1,1]^d} f(x) \, \mathrm{d}x &\approx \langle I, T f \rangle \notag \\
\int_{[-1,1]^d} f(x) \, \mathrm{d}x &\approx \langle I T^*, f \rangle
\end{align}
In words, we can compute the integral stencil by applying the adjoint transform operator to a vector representing the integral of the Chebyshev basis. In the case of Bravais lattices, the adjoint transform is the inverse up to scaling. For composite lattices, the adjoint and inverse operators are distinct but are still similar algorithms. Once an integral stencil is computed, applying it to a function is an $O(n)$ operation equivalent to a vector dot product.
\section{Results}
\subsection{Methodology}
It is common in the literature to present quality interpolation methods by demonstrating convergence for particular test functions. In this case, choosing a small subset of test functions is inadequate, because it is possible to contrive functions which will allow for any arbitrary basis to outperform any other. Simply by choosing functions whose Chebyshev coefficients decay according to $e^{-c \lVert k \rVert}$ for an arbitrary choice of norm on $k$, we can obtain virtually any convergence results we want. Trefethen's initial observation regarding Euclidean degree was made considering radially symmetric functions, so this would be a tempting route to consider. However, this could also be seen as a contrivance, since functions of interest may not be expected to have any particular symmetry in general.

\begin{figure}
\centering
\includegraphics[width=.99\textwidth]{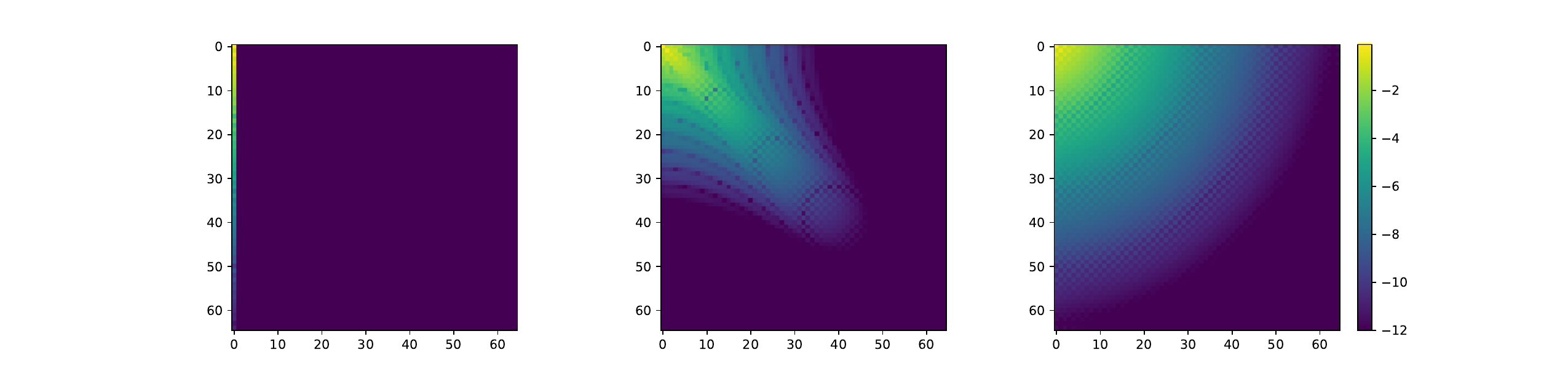}
\caption{Magnitudes of Chebyshev coefficients from left to right: $y$-axis aligned ridge function, the same function rotated $\frac{\pi}{4}$ radians, $L_2$ norm for this function over a full revolution}
\label{fig:ridge_func}
\end{figure}
The method we will use follows from considering the cubic domain as a subset of some larger domain, such as a patch on a surface or a subdomain in a domain decomposition scheme. In this case, one could consider different positions and orientations of a given cube with respect to a function defined on the full domain. Then we can consider the expected error over different orientations of a test function, effectively averaging the area over a family of equally representative test functions rather than a single one. 
If we consider \cref{fig:ridge_func}, the plot depicts the magnitude of polynomial coefficients for an analytic function which varies only in one direction. The orientation of this function changes the distribution of coefficients, but the Euclidean degree bound remains mostly consistent. This property is suggested by the proof in \cite{trefEucDegree} which proves a Euclidean degree coefficient bound for individual functions, but which can be readily generalized for families of functions related by rotation about the center of a cubic domain.

\subsection{Integration}
\begin{figure}
\centering
\includegraphics[width=.95\textwidth]{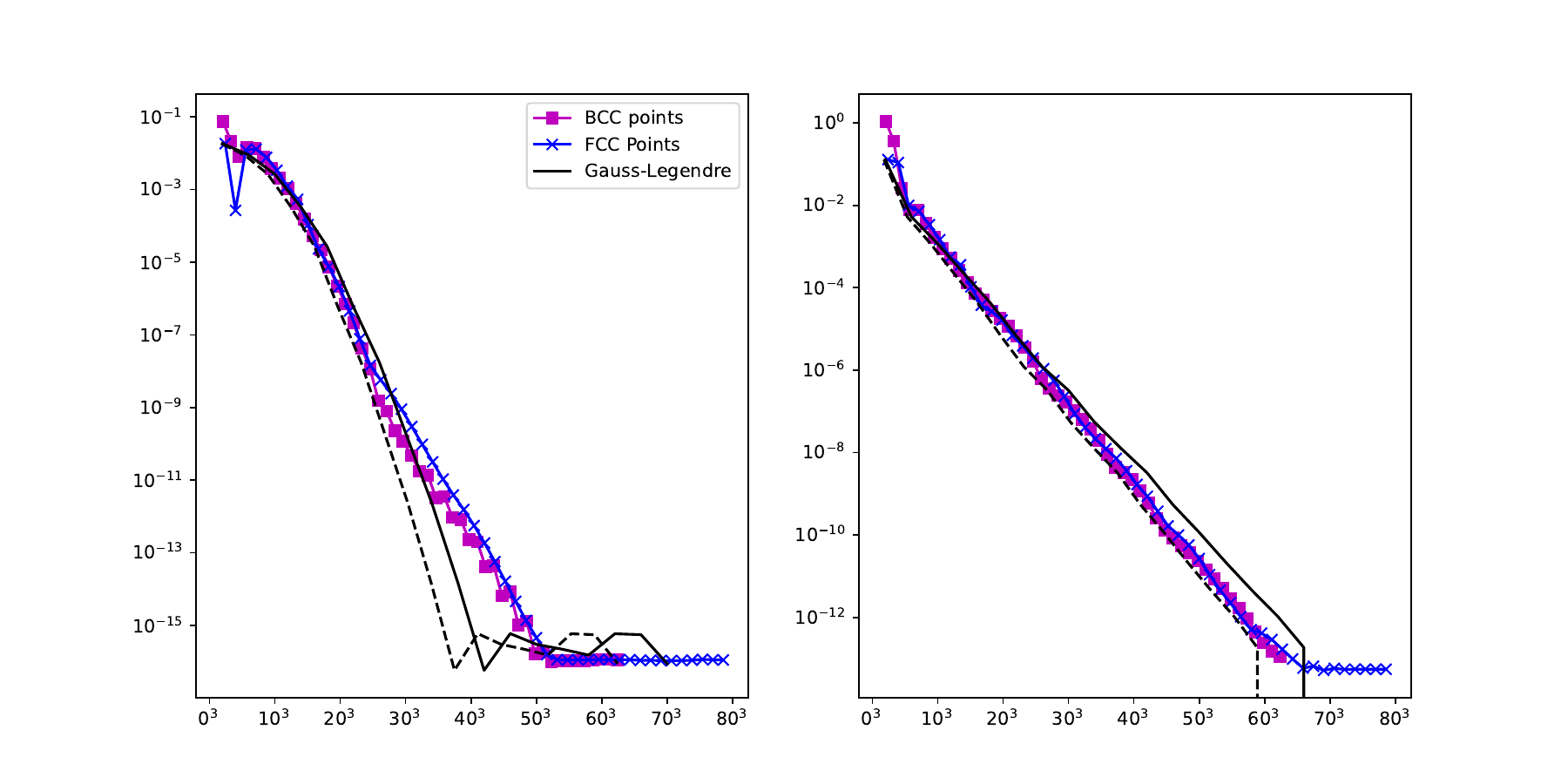}
\includegraphics[width=.45\textwidth]{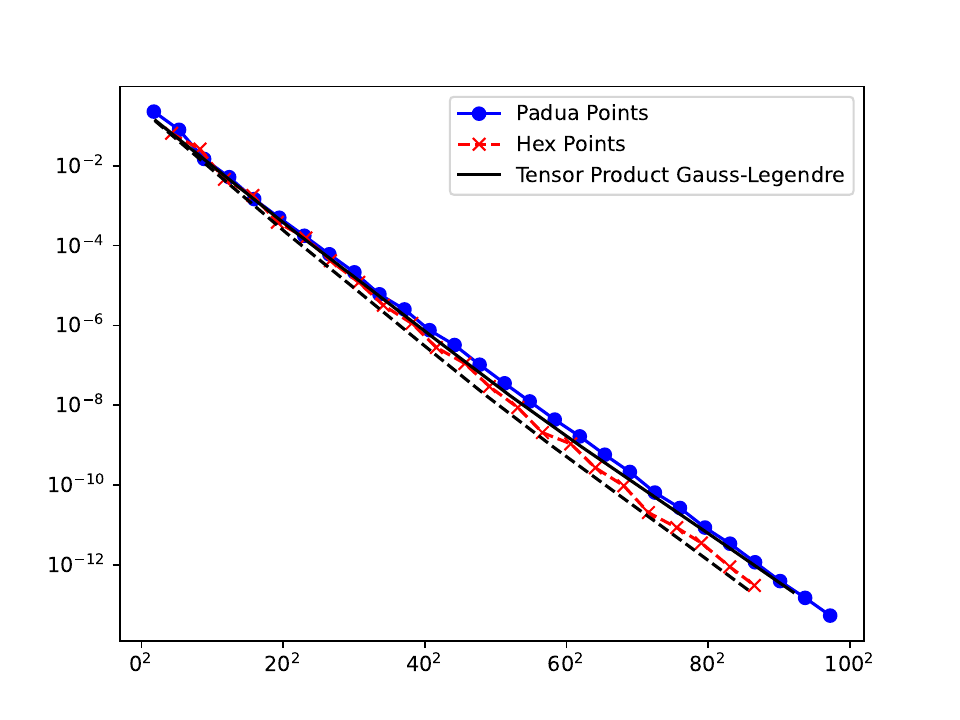}

\caption{Integration results comparing a few non-tensor product lattices to Gauss-Legendre in 3 dimensions (top) and 2 dimensions (bottom). Note that the x-scales are not linear. Dotted lines denote the expected rate of convergence for the most efficient Euclidean degree lattice based on tensor product Gauss-Legendre.}\label{fig:int_results}
\end{figure}
It is not common to use Clenshaw-Curtis quadrature in one dimension due to the availability of Gauss-Legendre quadrature. Gauss-Legendre has the guarantee of accurately integrating polynomials up to order $2n - 1$ with $n$ evaluation points, where Clenshaw-Curtis cannot only gets order $n - 1$ polynomials exactly. This would seem to indicate that Gauss-Legendre should wildly outperform Clenshaw-Curtis, but this is not always seen in practice~\cite{trefExactnessQuadrature}.

There are theoretical results explaining why Clenshaw-Curtis can achieve order $2n$ convergence in some circumstances~\cite{trefGQorCC}. Numerical experiments in 1d show that Clenshaw-Curtis exhibits an elbow phenomenon at which the rate of convergence seems to suddenly reduce by half. This elbow often occurs at high enough resolution as to have little significance for practical purposes. This same phenomenon can be seen in \cref{fig:int_results} in the case of 3d Clenshaw-curtis integrating an entire function. The elbow occurs for both of the non-Cartesian Chebyshev lattices when the relative error is of order between $10^{-7}$ and $10^{-8}$. For integrands that are merely analytic, this elbow may not occur at all before an integral reaches machine precision.

In \cref{fig:int_results}, we see comparisons of non-tensor product Clenshaw-Curtis and Gauss-Legendre. Integral errors are subject to wide variation, so these plots were made by averaging the error over a single asymmetric integral performed after many unitary coordinate transformations. From top to bottom, the integrands were of the form: $e^{-\|x - a\|^2}, \frac{1}{1 + \|x - a\|},$. This is intended to indicate the performance with entire and analytic functions. In each plot, there is a dotted line which indicates the theoretical performance of a lattice with the highest known Euclidean degree efficiency performing as well as the Gauss-Legendre stencil of the same degree. In two dimensions, achieving this rate of convergence would reduce the number of points needed to achieve a given error by about $12 \%$. In three dimensions, $29\%$.

We can see that Euclidean degree basis efficiency has an impact on the performance. In most cases, the Padua points converge at a similar rate to Gauss-Legendre, as both have the same Euclidean degree at similar point counts. The Hex, BCC, and FCC lattices often achieved higher accuracy at moderate to high resolution, usually converging at close to the rate that would be expected for a Gauss-Legendre stencil with the same Euclidean degree efficiency. However, for low resolutions we see that Gauss-Legendre tends to perform better. This is particularly troublesome for the hex points which require high resolution to get close to their expected level of efficiency.

\subsection{Interpolation}
\begin{figure}
\centering
\includegraphics[width=.95\textwidth]{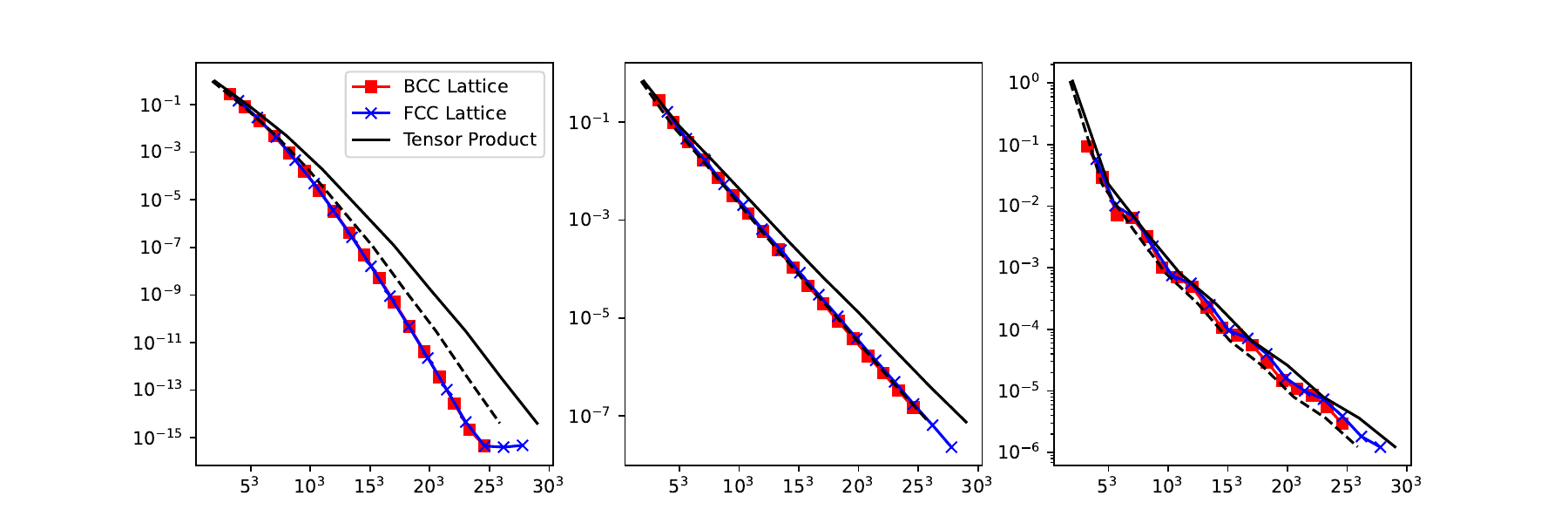}
\includegraphics[width=.95\textwidth]{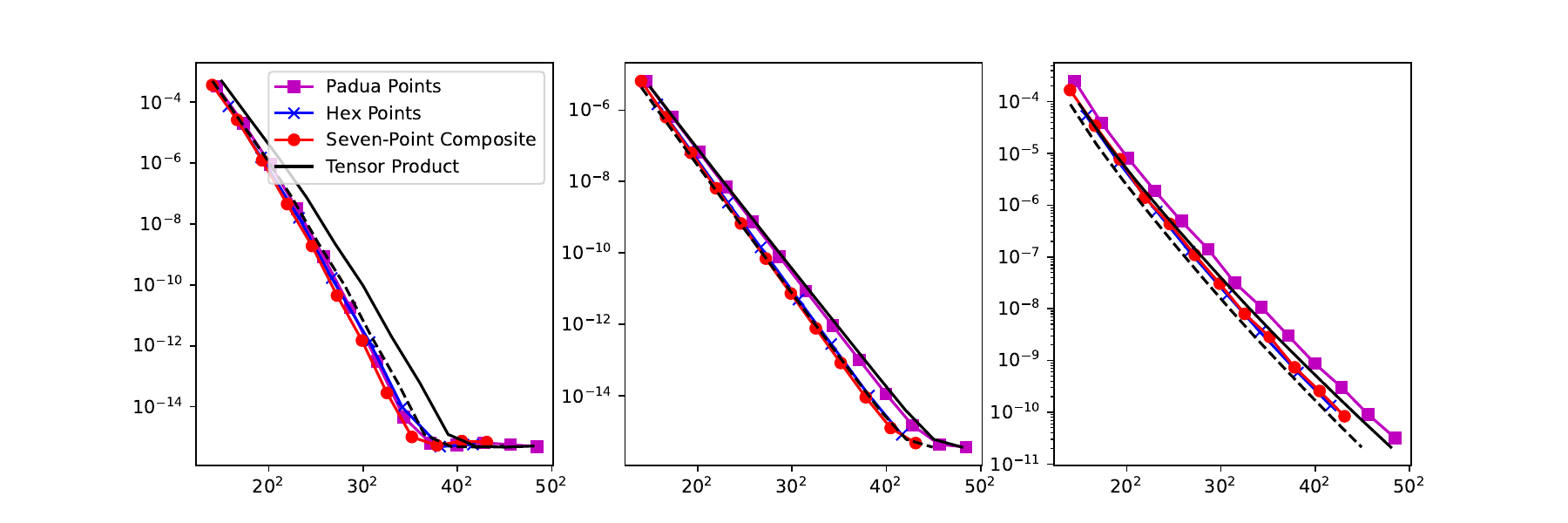}
\caption{Calculated $L_2$ error comparing the most efficient lattices in 2 and 3 dimensions to the tensor product lattices. Dotted line shows the expected rate of convergence based on Euclidean degree efficiency. From left to right, the chosen test functions are entire, analytic, and $C_\infty$ with an essential singularity in the domain.\label{fig:interp_results}}
\end{figure}
Relative to integration, interpolation produces much more consistent errors. There is also no method like Gauss-Legendre that might be expected to perform better than Chebyshev methods for interpolation. For this reason, our interpolation results are much more straightforward. We see that for the average case $L_2$ error, Euclidean degree efficiency is typically the most important factor. The convergence plots we have attached here show convergence for an off-center Gaussian and Runge function rotated and averaged. The resulting convergence plots conform extremely closely to expectation except for very low resolutions.

We see an interesting phenomenon occur when comparing the rates of convergence for different functions. In \cref{fig:interp_results}, we see that the efficient 3d lattices converge much faster than expected for an entire function, but at very nearly the expected rate for an analytic function. Something similar occurs in 2d, and by noting the performance of the Padua points in this case we can see what is happening. For functions with fast fall-off in their coefficients, we see that convergence is related more closely to total degree than Euclidean degree efficiency, while for more difficult to resolve functions, Euclidean degree efficiency is a good indicator of the convergence rate. In the most extreme case tested, interpolating a $C_\infty$ function, the Cartesian lattice outperforms expectation. 

This suggests that while Euclidean degree efficiency is a good indicator for many practical cases, it is not a tight bound on the error in every case.

\section{Future Work}
This project leaves several avenues open for development. Currently, we have only implemented one composite lattice in two dimensions. A viable 3 dimensional composite could produce a meaningful improvement over the best 3d Bravais lattices. To find a viable composite we would need to take a survey of Brillouin zones of Cartesian lattices in 3 dimensions with differing aspect ratios. Once a good candidate is found, we could find a good stencil and implement it.

Another direction is to improve integration accuracy, particularly for low resolutions. Although it has not been included in this paper, we have seen some promise from integral stencils based on Chebyshev polynomials of the second kind. The method for doing this is a simple extension of the methods shown here along with a linear rescaling procedure to improve aliasing errors. This method could allow Clenshaw-Curtis based integration to compete with Gauss-Legendre at low resolutions, especially in 3 dimensions.

Although the preliminary results shown here show potential for this method, one issue in applying it is the lack of optimization in its current implementation. Many of the necessary procedures are currently coded in pure python, which is well known to produce slower results relative to lower level languages. A re-implementation of the current codebase to avoid large python loops would improve the overall code performance and allow for direct speed comparisons to existing polynomial methods.
\section{Acknowledgements}
We sincerely thank Travis Askham for his contributions both in inspiring our initial research into total degree efficient lattices and for his feedback throughout the writing process.
\bibliographystyle{unsrt}  
\bibliography{references}

\end{document}